\documentclass[12pt,a4paper]{article}
\usepackage[english]{babel}
\usepackage{textcase}

\usepackage{amsmath,amssymb,xspace,color,eucal,multirow}
\usepackage[amsthm,amsmath,thmmarks]{ntheorem}
\usepackage{accents}

\usepackage[
  pdftitle={An upper bound on binomial coefficients 
    in the de Moivre~-- Laplace form}, 
  pdfauthor={Sergey Agievich},
  bookmarks=true,
  pdfpagemode=UseOutlines,
  pdfstartview=FitH,
  linkcolor=black,
  citecolor=black]{hyperref}

\usepackage[
  style=numeric,
  backend=biber,
  url=true]{biblatex}
\bibliography{binom}

\AtEveryBibitem{
  \clearfield{isbn}
  \clearfield{issn}
  \clearfield{doi}
  \clearfield{month}
}

\renewbibmacro{in:}{
  \ifentrytype{article}{}{\printtext{\bibstring{in}\intitlepunct}}}

\DeclareFieldFormat
  [article,inbook,incollection,inproceedings,patent,thesis,unpublished]
  {title}{#1\isdot}

\renewbibmacro*{volume+number+eid}{%
  \printfield{volume}
  \setunit*{\addnbspace}
  \printfield{number}
  \setunit{\addcomma\space}
  \printfield{eid}}
\DeclareFieldFormat[article]{volume}{\mkbibbold{#1}}
\DeclareFieldFormat[article]{number}{\mkbibparens{#1}}

\DeclareFieldFormat{sentencecase}{\MakeSentenceCase*{#1}}
\renewbibmacro{title}{%
    \ifthenelse{\iffieldundef{title}\AND\iffieldundef{subtitle}}{}
        {\ifthenelse{\ifentrytype{article}\OR\ifentrytype{inbook}%
            \OR\ifentrytype{incollection}\OR\ifentrytype{inproceedings}}
            {\printtext[title]{%
                \printfield[sentencecase]{title}%
                \setunit{\subtitlepunct}%
                \printfield[sentencecase]{subtitle}}%
                \newunit}%
            {\printtext[title]{%
                \printfield[titlecase]{title}%
                \setunit{\subtitlepunct}%
                \printfield[titlecase]{subtitle}}%
                \newunit}}%
    \printfield{titleaddon}}

\providecommand{\FF}{\mathbb{F}}
\providecommand{\RR}{\mathbb{R}}
\providecommand{\ZZ}{\mathbb{Z}}
\providecommand{\calF}{\mathcal{F}}
\providecommand{\calB}{\mathcal{B}}

\DeclareMathOperator\arctanh{arctanh}

\newcommand{\PrSymbol}{\mbox{\sffamily\upshape\bfseries P}}
\renewcommand{\Pr}[2][]{\PrSymbol_{#1}\left\{#2\right\}}
\renewcommand{\vec}[1]{{\bf #1}}

\newcommand{\hatup}[1]{\accentset{\wedge}{#1}}
\newcommand{\boxup}[1]{\accentset{\square}{#1}}

\theoremstyle{plain}
\theorembodyfont{\normalfont\itshape}
\theoremseparator{.}

\renewtheorem*{theorem*}{Theorem}

\renewtheorem*{lemma*}{Lemma}
\newtheorem{proposition}{Proposition}

\theoremstyle{plain}
\theorembodyfont{\normalfont\upshape\small}
\theoremseparator{.}
\newtheorem{example}{Example}

\begin{document}
\begin{sloppypar}

\title{An upper bound on binomial coefficients 
in~the~de~Moivre--Laplace form\footnote{
Appeared in {\it Journal of the Belarusian State University.
Mathematics and Informatics}, 2022; 1:66--74 (in~Russian).}
}

\author{Sergey~Agievich\\
\small Research Institute for Applied Problems of Mathematics and 
Informatics,\\[-0.6ex] 
\small Belarusian State University, Minsk, Belarus\\[-0.6ex]
\small \url{agievich@{bsu.by,gmail.com}}
}

\date{}

\maketitle

\begin{abstract}
We suggest an upper bound on binomial coefficients that holds over the 
entire parameter range and whose form repeats the form of the de Moivre--Laplace 
approximation of the symmetric binomial distribution.
Using the bound, we estimate the number of continuations of a given Boolean 
function to bent functions, investigate dependencies into the Walsh--Hadamard 
spectra, obtain restrictions on the number of representations as sum of 
squares of integers bounded in magnitude.
\end{abstract}


\section{Results}\label{PRELIM}

The de Moivre--Laplace theorem for the symmetric binomial distribution 
can be written as the following bound on the binomial coefficients:
$$
\binom{n}{k}=\frac{2^n}{\sqrt{\pi n/2}}\exp\left(-\frac{2(k-n/2)^2}{n}
\right)(1+O(1/\sqrt{n})).
$$
The bound holds for $n\to\infty$ and $|k-n/2|=O(\sqrt{n})$, 
that is, in the so-called central region.

The fact that the bound is asymptotic and valid only in the 
central region makes it difficult to apply in certain cases 
(some of them are considered in this paper).
There are known non-asymptotic bounds that hold for wider ranges of parameters.
For example,
\begin{align*}
\left(\frac{n}{k}\right)^k 
&\leq \binom{n}{k} 
\leq\left(\frac{en}{k}\right)^k,\quad
1\leq k\leq n,\\
\intertext{or, denoting $H_2(x)=-x\log_2 x - (1-x)\log_2(1-x)$,}
\frac{2^{n H_2(k/n)}}{\sqrt{8 k (1-k/n)}}
&\leq \binom{n}{k} 
\leq \frac{2^{n H_2(k/n)}}{\sqrt{2\pi k (1-k/n)}},\quad
1\leq k\leq n-1
\end{align*}
(see~\cite{Odl95} and \cite[Chapter~10, Lemma~7]{McWSlo78} respectively).
However, these bounds are either not accurate enough or their forms are not 
convenient enough. 

We suggest an upper bound on the binomial coefficients that holds over the 
entire parameter range and preserves the de Moivre--Laplace form.

\begin{theorem*}\label{Th.PRELIM.1}
For a positive integer~$n$ and $k\in\{0,1,\ldots,n\}$, it holds that
$$
\binom{n}{k}\leq \frac{2^n}{\sqrt{\pi n/2}}
\exp\left(-\frac{2(k-n/2)^2}{n}+\frac{23}{18n}\right).
$$
\end{theorem*}

When deriving the bound, we follow the approach of the paper~\cite{Sza16} which 
in turn extends previous lines of research as explained in the paper.

In Sections~\ref{BENT}~--- \ref{SQUARES}, we discuss the use of the proposed bound.
More specifically, in Section~\ref{BENT} we estimate the number of 
continuations of a given Boolean function to bent functions,
in Section~\ref{LATIN} we investigate dependencies into the Walsh--Hadamard 
spectra,
in Section~\ref{SQUARES} we obtain restrictions on the number of 
representations as sum of squares of integers bounded in magnitude.
Section~\ref{PROOF} contains the proof of Theorem.

\section{Continuations to bent functions}\label{BENT}

Let $\FF_2$ be the field of 2 elements ($0$ and $1$), $\FF_2^n$ be the
$n$-dimensional vector space over~$\FF_2$, and $\calF_n$ be the set of Boolean
functions in $n$ variables, that is, the functions from $\FF_2^n$ to~$\FF_2$.
A function $f\in\calF_n$ is uniquely described by its Walsh--Hadamard 
spectrum (spectral function)  
$$
\hatup{f}(\vec{u})=
\sum_{\vec{x}\in\FF_2^n}\chi(f(\vec{x})+\vec{x}\cdot\vec{u}),\quad
\vec{u}\in\FF_2^n.
$$
Here $\chi$ is the non-trivial additive character of $\FF_2$: $\chi(a)=(-1)^a$,
the dot denotes the inner product of vectors.

A spectrum~$\hatup{f}$ satisfies Parseval's identity:
$$
\sum_{\vec{u}\in\FF_2^n}\hatup{f}(\vec{u})^2=2^{2n}.
$$
Due to this identity, $\max_\vec{u}|\hatup{f}(\vec{u})|\geq 2^{n/2}$.
If the lower bound is attained (this is only possible when $n$ is even), 
then~$f$ is called a \emph{bent function}~\cite{Rot76}.
Let $\calB_n$ be the set of bent functions in $n$ variables.

Bent functions are ideal objects in several contexts of coding theory, 
cryptography and combinatorics.
Despite intensive research, bent functions remain difficult to study, 
there are many open problems related to them. 
One of these problems is to estimate the number of bent functions both from 
below and from above.
In~\cite{Agi20} it is proposed to obtain upper bounds by estimating the number 
of continuations of a Boolean function to bent functions.
Next we recall and detail the approach of~\cite{Agi20}.

Let~$k<n$. 
A function $f\in\calF_n$ is a \emph{continuation} of~$g\in\calF_k$ if
$$
g(y_1,\ldots,y_k)=f(\underbrace{0,\ldots,0}_{n-k},y_1,\ldots,y_k). 
$$
In other words, $f$ is a continuation of~$g$ if $g$ is a \emph{restriction} 
of~$f$ to the affine plane $E=\{(0,\ldots,0,y_1,\ldots,y_k)\colon y_i\in\FF_2\}$. 
The actual choice of~$E$ here is unimportant, any other plane of dimension~$k$ 
can be used instead.

Let $\calB_n(g)$ be the set of all functions~$f\in\calB_n$ that continue~$g$. 
To obtain an upper bound on $|\calB_n(g)|$ we use the representation of bent 
functions by bent rectangles. This representation is introduced 
in~\cite{Agi02}. Let us recall it. 

Let $f\in\calF_n$ and $n=m+k$, where $m$ and~$k$ are positive integers. 
Consider all possible restrictions of~$f$ on the planes parallel to~$E$:
$$
f_\vec{u}(\vec{y})=f(\vec{u},\vec{y}),\quad
\vec{u}\in\FF_2^m,\quad
\vec{y}\in\FF_2^k.
$$
Let us pass from the restrictions $f_\vec{u}$ to their spectra $\hatup{f}_\vec{u}$ 
and then construct the function
$$
\boxup{f}(\vec{u},\vec{v})=\hatup{f}_\vec{u}(\vec{v}),\quad
\vec{u}\in\FF_2^m,\quad
\vec{v}\in\FF_2^k.
$$
We call $\boxup{f}$ a \emph{rectangle} of $f$.
By construction, the restrictions $\boxup{f}(\vec{u},\vec{v})$ to $\vec{v}$ 
(\emph{rows}) are spectral functions. If additionally the restrictions  
$\boxup{f}(\vec{u},\vec{v})$ to $\vec{u}$ (\emph{columns}) multiplied by 
$2^{(m-k)/2}$ are also spectral functions, then $\boxup{f}$ 
is called a \emph{bent rectangle}.
In~\cite{Agi02} it is proved that $f$ is bent if and only if $\boxup{f}$ is bent.

In terms of bent rectangles, the problem of estimating the cardinality of 
$\calB_n(g)$ reduces to estimating the number of bent rectangles $\boxup{f}$ 
whose first row is fixed:
$$
\boxup{f}(\vec{0},\vec{v}) = \hatup{g}(\vec{v}).
$$

Note that if $k>n/2$, then there exist functions $g$ that cannot be 
continued: $|\calB_n(g)|=0$. 
An example is a function that takes exactly $2^{k-1}+1$ zero values and 
therefore~$\hatup{g}(\vec{0})=2$.
The normalized column $2^{(m-k)/2}\boxup{f}(\vec{u},\vec{0})$ takes odd or 
fractional value
$$
2^{(m-k)/2}\boxup{f}(\vec{0},\vec{0})=2^{(m-k)/2}\hatup{g}(\vec{0})=2^{(m-k)/2+1}
$$ 
and cannot be a spectral function. 
Hence $\boxup{f}$ cannot be a bent rectangle.

For $k\leq n/2$ the situation changes.

\begin{proposition}\label{Pr.BENT.1}
If~$n$ is even, then for any Boolean function $g$ in $k\leq n/2$ variables
it holds that $|\calB_n(g)|>0$ and
$$
\log_2|\calB_n(g)|\leq {2^n\left(1-\gamma_{2^{n-k}}\right)},
$$
where
$$
\gamma_M=
\frac{\log_2 e+\log_2\pi + \log_2 M - 1}{2M}-\frac{23\log_2 e}{18M^2}.
$$
\end{proposition}
\begin{proof}
Let us prove that $|\calB_n(g)|\neq 0$.
It is sufficient to consider the case $k=n/2$. The rectangle
$$
\boxup{f}(\vec{u},\vec{v})=\hatup{g}(\vec{u}+\vec{v}),\quad
\vec{u},\vec{v}\in\FF_2^k,
$$
implements the biaffine construction from~\cite{Agi08} and is therefore bent.
The function~$f$ corresponding to~$\boxup{f}$ is also bent. 
Moreover, the first ($\vec{u}=\vec{0}$) row of $\boxup{f}$ coincides
with~$\hatup{g}$ and $f$ is a continuation of~$g$.
In whole, $f\in\calB_n(g)$ and $\calB_n(g)$ is non-empty.

Let us proceed to find an upper bound on $|\calB_n(g)|$.
We have to estimate the number of bent rectangles $\boxup{f}(\vec{u},\vec{v})$  
such that $\boxup{f}(\vec{0},\vec{v})=\hatup{g}(\vec{v})$. 
Denote $M=2^m$, $K=2^k$, $s_\vec{v}=2^{(m-k)/2}\hatup{g}(\vec{v})$.

Consider the columns of $\boxup{f}$ multiplied by $2^{(m-k)/2}$.
They are spectral functions
$$
\hatup{g}_\vec{v}(\vec{u})=2^{(m-k)/2}\boxup{f}(\vec{u},\vec{v}),\quad
\vec{u}\in\FF_2^m,\quad
\vec{v}\in\FF_2^k,
$$
which correspond to the functions $g_\vec{v}\in\calF_m$. 
According to the restrictions on $\boxup{f}$,
$$
\hatup{g}_\vec{v}(\vec{0})=
2^{(m-k)/2}\boxup{f}(\vec{0},\vec{v})=
2^{(m-k)/2}\hatup{g}(\vec{v})=
s_\vec{v}.
$$

There are $2^M$ ways to choose $g_\vec{v}$ and exactly 
$$
\binom{M}{(M+s_\vec{v})/2}
$$
of them yield the equality $\hatup{g}_\vec{v}(\vec{0})=s_\vec{v}$.
Therefore, the desired number of continuations
(the number of suitable bent rectangles)
$$
|\calB_n(g)|\leq\prod_{\vec{v}\in\FF_2^k}\binom{M}{(M+s_\vec{v})/2}.
$$
Taking the logarithm of both parts of this inequality and using the bound of 
Theorem, we obtain
$$
\log_2|\calB_n(g)|\leq\sum_{\vec{v}\in\FF_2^k}
\left(M-\alpha_M s_\vec{v}^2 - \beta_M\right).
$$
Here
$$
\alpha_M = \frac{\log_2 e}{2M},\quad
\beta_M  = \frac{1}{2}(\log_2\pi + \log_2 M - 1) - \frac{23\log_2 e}{18M}.
$$

Using the equality
$$
\sum_{\vec{v}\in\FF_2^k} s_\vec{v}^2 = 
2^{m-k}\sum_{\vec{v}\in\FF_2^k}\hatup{g}(\vec{v})^2 =
2^{m-k}\cdot 2^{2k} = MK,
$$
we finally get
$$
\log_2|\calB_n(g)|\leq MK(1-\alpha_M -\beta_M/M) = MK(1-\gamma_M).
$$
That was to be proven.
\end{proof}

In the proof, we used the following form of the bound of Theorem:
$$
\binom{M}{(M+s)/2}\leq 2^{M-\alpha_M s^2-\beta_M}.
$$
From the coefficients $\alpha_M$ and~$\beta_M$ we determined the 
quantity~$\gamma_M=\alpha_M+\beta_M/M$.
The larger~$\gamma_M$, the more accurate the estimate of $|\calB_n(g)|$.
It turns out that $\alpha_M$ and~$\beta_M$ can be adjusted so that $\gamma_M$ 
increases, but the bound on the binomial coefficients remains valid. 

For small~$M$, the optimal tuples $(\alpha_M,\beta_M,\gamma_M)$ can be found by 
solving the linear programming problem. The solutions are presented in the table below.
The quantities $\gamma_M$ in the last column of the table can be used in 
Proposition~\ref{Pr.BENT.1} instead of those specified there.

$$
\begin{array}{|c|c|c|c|}
\hline
M  &  \alpha_M & \beta_M & \gamma_M\\
\hline
\hline
2  &  1/2 & 1 & 3/4\\
4  &  1/6 & 4/3 & 1/2\\
8  &  1/12 & 14/3-\log_2 7 & 
   2/3-(\log_2 7)/8\approx 0.3157\\
\hline
\end{array}
$$

\section{Latin dependencies}\label{LATIN}

Let~$\ZZ^N$ be the set of $N$-tuples of integers, 
$\Omega$ be a finite subset of $\ZZ^N$,
and $p$ be a probability distribution over $\Omega$.
Let $\vec{a}=(a_1,\ldots,a_N)$ be a random tuple of~$\Omega$ following~$p$ and
let $p_i$ be the marginal distribution of the $i$th coordinate of the tuple:
$p_i(x)=\Pr{a_i=x}$, $i=1,\ldots,N$.

The degree of dependence between coordinates of $\vec{a}$ can be estimated as follows.
\begin{enumerate}
\item
Choose tuples $\vec{a}^1,\ldots,\vec{a}^N$ independently at random according to 
the distribution~$p$.
\item
Construct a tuple $\vec{b}=(b_1,\ldots,b_N)$ in which $b_i$ is the $i$th 
coordinate of $\vec{a}^i$. 
\item
Determine the degree of dependence: $L(p) = \Pr{\vec{b}\in\Omega}$.
\end{enumerate}


It is convenient to imagine that the tuples $\vec{a}^i$ form the rows of a 
matrix of order~$N$ and then $\vec{b}$ is the diagonal of the matrix.
The probability $\Pr{\vec{b}\in\Omega}$ characterizes the fulfillment of a 
constraint on the diagonal provided that all rows satisfy this constraint.
Similar constraints (on rows, columns, sometimes on diagonals) are imposed in 
Latin squares. That is why we call the quantity $L(p)$ the degree of the 
\emph{Latin} dependence.

The quantity~$L(p)$ is the probability of a successful ``assembly'' of an 
element of~$\Omega$ from ``scattered'' coordinates with the 
distributions~$p_1,\ldots,p_N$.
As the dependence between the coordinates of~$\vec{a}$ increases, we should 
expect a decrease in the probability~$L(p)$.
The maximum value of $L(p)=1$ is achieved when the coordinates of~$\vec{a}$ 
are independent.

The degree of the Latin dependence can be computed as follows:
$$
L(p)=\sum_{(b_1,\ldots, b_N)\in\Omega}\prod_{i=1}^N p_i(b_i).
$$

\begin{example}
Let $\Omega$ consist of permutations of the numbers from $1$ to $N$ 
and let $p$ be the uniform distribution over~$\Omega$.
Then $p_i(x) = 1 / N$ if $x\in\{1,\ldots,N\}$ and $p_i(x)=0$ otherwise.
Thus,
$$
L(p)=\frac{N!}{N^N} \approx \frac{\sqrt{2\pi N}}{e^N}.
$$
We can speak of the \emph{exponential} dependence meaning that $1/L(p)$ 
grows exponentially with $N$.
\end{example}

\begin{example}
Let $N$ be even and $p$ assign the probability $1/\binom{N}{N/2}$ 
to each of $(0,1)$-tuples of length $N$ with exactly $N/2$ units.
Then~$p_i(x)=1/2$ for $x\in\{0,1\}$ and
$$
L(p)=\binom{N}{N/2} 2^{-N}\approx \sqrt{\frac{2}{\pi N}}.
$$
We can speak of the \emph{power} dependence or, more precisely, ``\emph{square root}'' 
dependence.
\end{example}

Let us show how to use Theorem to estimate the degree of the Latin dependence 
in the Walsh--Hadamard spectra (see Section~\ref{BENT}). 

\begin{proposition}\label{Pr.LATIN.1}
Let $\Omega$ consist of tuples of values of spectral functions $\hatup{f}$ 
corresponding to all possible $f\in\calF_n$.
Let~$p$ be the uniform distribution over~$\Omega$.
Then
$$
L(p)\leq
\exp\left(\frac{23}{18}\right)
\left(\frac{8}{\pi e N}\right)^{N/2},\quad
N=2^n.
$$
\end{proposition}

\begin{proof}
Since there are $2^N$ functions $f$ and the mapping $f\mapsto\hatup{f}$ is bijective, 
$|\Omega|=2^N$.
Elements of $\Omega$ are $N$-tuples of even numbers bounded in magnitude by~$N$. 
Marginal distributions of coordinates of the tuples:
$$
p_i(x)=2^{-N}\binom{N}{(N+x)/2},\quad
x\in\{-N,-N+2,\ldots,N\}.
$$
The degree of dependence between coordinates:
$$
L(p)=\sum_{(b_1,\ldots,b_N)\in\Omega}\prod_{i=1}^N 2^{-N}\binom{N}{(N+b_i)/2}.
$$
Note that due to Parseval's identity, $\sum_i b_i^2 = N^2$.

Applying Theorem, we obtain
\begin{align*}
L(p)
&
\leq 2^N \max_{(b_1,\ldots,b_N)\in\Omega}
\prod_{i=1}^N\sqrt{\frac{2}{\pi N}}\exp\left(-\frac{b_i^2}{2N}+\frac{23}{18N}\right)=\\
& =
2^N \left(\frac{2}{\pi N}\right)^{N/2}
\exp\left(-\frac{N}{2}+\frac{23}{18}\right)
=
\exp\left(\frac{23}{18}\right)
\left(\frac{8}{\pi e N}\right)^{N/2}.
\end{align*}
That was to be proven.
\end{proof}

As we see, the degree of dependence in the Walsh--Hadamard spectra 
is asymptotically higher than in the permutations. In the case of spectra, 
we can speak of the \emph{factorial} dependence.

\section{Representations as sum of squares}\label{SQUARES}

Let $r_{s,n}(N)$ be a number of representations of a positive integer $N$ 
as sum of squares of~$s$ integers, each bounded in magnitude by~$n$:
$$
r_{s,n}(N)=\left|\left\{(a_1,\ldots,a_s)\in\ZZ^s\colon
\textstyle\sum_{i=1}^s a_i^2=N,\ 
|a_i|\leq n\right\}\right|. 
$$

The function $r_{s,n}(N)$ and especially the function $r_s(N)$, in which the 
restrictions on $|a_i|$ are removed, have long been studied in number theory.
One of the results is the following asymptotic integral bound on $r_s(N)$:
$$
\sum_{N=1}^R r_s(N)=
\frac{(\pi R)^{s/2}}{\Gamma(s/2+1)}+O(R^{(s-1)/2}),\quad
R\to\infty
$$
(see for example~\cite{Tak18}).
For $s=2$, this is known as Gauss' circle theorem.

Using our Theorem, we obtain an integral bound on $r_{s,n}(N)$.

\begin{proposition}\label{Pr.SQUARES.1}
It holds that
$$
\sum_{N=0}^{sn^2}
r_{s,n}(N)\exp\left(-\frac{N}{n}\right)
\geq
(\pi n)^{s/2}
\exp\left(-\frac{23 s}{36 n}\right).
$$
\end{proposition}

\begin{proof}
A tuple $(a_1,\ldots,a_s)$ is an integral point of the $s$-dimensional 
real space $\RR^s$. 
The point lies within $s$-dimensional cube with side~$2n$ and simultaneously 
on a circle of radius~$\sqrt{a_1^2+\ldots+a_s^2}$ centered at the origin.
Any point within the cube lies on a circle of some radius and the square of 
this radius is a non-negative integer not exceeding $sn^2$. 
Let us use this observation. 

Let $\xi_1,\ldots,\xi_s$ be random variables obtained by summing 
$2n$ independent random variables taking the values $1$ and $-1$
with equal probability.
Then for the points $(a_1,\ldots,a_s)$ within the cube it holds that 
$$
\Pr{(\xi_1,\ldots,\xi_s)=(2a_1,\ldots,2a_s)}
=\prod_{i=1}^s 2^{-2n}\binom{2n}{n+a_i}
\leq
\frac{1}{(\pi n)^{s/2}}\exp\left(-\frac{N}{n}+\frac{23 s}{36 n}\right).
$$
Here $N=\sum_i a_i^2$ is a square of a radius of a circle containing 
the point $(a_1,\ldots,a_s)$.

We have
\begin{align*}
1 &= \sum_{-n\leq a_1,\ldots,a_s\leq n}
\Pr{(\xi_1,\ldots,\xi_s)=(2a_1,\ldots,2a_s)}=\\
&=\sum_{N=0}^{sn^2}\,\sum_{\genfrac{}{}{0pt}{1}
{-n\leq a_1,\ldots,a_s\leq n}{\sum a_i^2 = N}}
\Pr{(\xi_1,\ldots,\xi_s)=(2a_1,\ldots,2a_s)}\leq\\
&
\leq 
\sum_{N=0}^{sn^2}
r_{s,n}(N)
\frac{1}{(\pi n)^{s/2}}\exp\left(-\frac{N}{n}+\frac{23 s}{36 n}\right),
\end{align*}
from which the desired result follows.
\end{proof}

\section{Proof of Theorem}\label{PROOF}

\begin{lemma*}\label{L.PROOF.1}
For a positive integer~$n$ and $k=0,\pm 1,\ldots,\pm n$, it holds that
$$
\binom{2n}{n+k}\leq\frac{2^{2n}}{\sqrt{\pi n}}
\exp\left(-\frac{k^2}{n}+\frac{23}{36n}\right).
$$
\end{lemma*}
\begin{proof}
The bound obviously holds for $k=\pm n$.
For $n=1,2$, the bound is verified by direct calculations.
Since the binomial coefficients $\binom{2n}{n+k}$ and $\binom{2n}{n-k}$ 
coincide, it remains to deal with the case where $n\geq 3$ and $0\leq k<n$.

For this case, as shown in~\cite{Sza16},
$$
\log\left(\binom{2n}{n+k}2^{-2n}\right)\leq 
\log\frac{1}{\sqrt{\pi n}} - b_{k,n} - \frac{1}{9n},
$$
where
$$
b_{k,n}=n\left(
\left(1+\frac{k+1/2}{n}\right)\log\left(1+\frac{k}{n}\right)+
\left(1-\frac{k-1/2}{n}\right)\log\left(1-\frac{k}{n}\right)
\right).
$$
To complete the proof, we show that $b_{k,n}>k^2/n-c/n$, $c=23/36+1/9=3/4$.

Let us consider the function $f(k)=b_{k,n}-k^2/n$. 
In~\cite{Sza16}, it is represented as follows:
$$
f(k)=
-\frac{k^2}{2n^2}+
\frac{k^4}{2n^3}\left(\frac{1}{3}-\frac{1}{2n}\right)+
\frac{k^6}{2n^5}\left(\frac{1}{5}-\frac{1}{2n}\right)+\ldots.
$$
Let $k_0=\sqrt{3n/2}$. In the region $k\leq k_0$, the inequality
$$
f(k)>-\frac{k^2}{2n^2}\geq -\frac{k_0^2}{2n^2}=-\frac{c}{n}
$$
holds. 
Now it is sufficient to prove that $f(k)$, viewed as a real-variable function, 
increases for $k\in[k_0,n-1]$.

Denote $x=k/n$ and take the derivative:
\begin{align*}
f'(k)&=2\arctanh\left(\frac{k}{n}\right)-\frac{2k}{n}+\frac{k}{k^2-n^2}=\\
&=2\arctanh x-2x-\frac{x}{2n(1-x^2)}=\\
&=\frac{2x^3}{3}+\frac{2x^5}{5}+\frac{2x^7}{7}+\ldots-\frac{x}{2n(1-x^2)}>\\
&>\frac{2x^3}{3}+\frac{2x^5}{5}-\frac{x}{2n(1-x^2)}.
\end{align*}

The derivative of the last expression has the form:
$$
\frac{(1+x^2)(4n x^2(1-x^2)^2-1)}{2n(1-x^2)^2}.
$$
It is positive in the region $x\in[k_0/n,(n-1)/n]$ and therefore 
$f'(k)$ increases for $k\in[k_0,n-1]$. 
%
%
Moreover,
$$
f'(k_0)>\frac{2x_0^3}{3}-\frac{x_0}{2n(1-x_0^2)}=
x_0\left(\frac{1}{n}-\frac{1}{2n(1-3/(2n))}\right)\geq 0
$$
(we take into account that $n\geq 3$) and therefore $f(k)$ also increases.
\end{proof}

Let us now turn to the proof of Theorem.
It is sufficient to consider the case of odd $n=2m-1$ and $k\leq m$.
In this case,
\begin{align*}
\binom{2m-1}{k}&=\frac{2m-k}{2m}\binom{2m}{k}\leq\\
&\leq\left(1-\frac{k}{2m}\right)\frac{2^{2m}}{\sqrt{\pi m}}
\exp\left(-\frac{(k-m)^2}{m}+\frac{23}{36m}\right)\leq\\
&\leq\frac{2^{2m-1}}{\sqrt{\pi (m-1/2)}}
\exp\left(-\frac{(k-m+1/2)^2}{m-1/2}+\frac{23}{36(m-1/2)}+t(k)\right),
\end{align*}
where
\begin{align*}
t(k)&=-\frac{(k-m)^2}{m}+
  \frac{(k-m+1/2)^2}{m-1/2}+\log 2 + \log\left(1-\frac{k}{2m}\right)=\\
&=\log\left(2-\frac{k}{m}\right)-\frac{2m^2-m-2k^2}{2m(2m-1)}.
\end{align*}

To complete the prove, we show that $t(k)\geq 0$ for $k\in[0;m]$.

The derivative 
$$
t'(k)=\frac{2k}{2m^2-m}-\frac{1}{2m-k}=\frac{2(m-k)^2-m}{(2m^2-m)(2m-k)}.
$$
Therefore, the minimum of $t(k)$ is reached at the point $k_0=m-\sqrt{m/2}$, 
and this minimum is
$$
t(k_0)=\log\left(1+\frac{1}{\sqrt{2m}}\right)-\frac{1}{1+\sqrt{2m}}.
$$
Denoting $x=\sqrt{2m}$, we obtain
$$
t(k_0)\geq \frac{1}{x}-\frac{1}{2x^2}-\frac{1}{1+x}=\frac{x-1}{2x^2(x+1)}>0.
$$
That was to be shown.

\printbibliography

@inproceedings{Agi02,
  author = {Agievich, S.},
  title = {On the representation of bent functions by bent rectangles},
  booktitle = {Probabilistic Methods in Discrete Mathematics: Fifth International 
    Conference (Petrozavodsk, Russia, June 1--6, 2000)},
  address = {Utrecht, Boston},
  publisher = {VSP}, 
  year = {2002}, 
  pages = {121--135},
  language = {english},
}

@inproceedings{Agi08,
  author = {Agievich, S.},
  title = {Bent Rectangles}, 
  booktitle = {Proceedings of the NATO Advanced Study Institute on Boolean 
    Functions in Cryptology and Information Security (Moscow, September 8--18, 
    2007)},
  address = {Amsterdam},
  publisher = {IOS Press}, 
  year = {2008}, 
  pages = {3--22},
  language = {english},
}

@article{Agi20,
  author = {Agievich, S.},
  title = {On the continuation to bent functions and upper bounds on 
    their number}, 
  journal = {Prikl. Diskr. Mat. Suppl.}, 
  volume = {13},
  year = {2020}, 
  pages = {18--21},
  note = {In Russian},
  language = {english},
}

@book{McWSlo78,
  title = {The Theory of Error-Correcting Codes},
  author = {MacWilliams, F.J. and Sloane, N.J.A.},
  year = {1978},
  edition = {2nd},
  publisher = {North-holland Publishing Company},
  language = {english},
}

@inbook{Odl95,
  author = {Odlyzko, A.M.},
  title = {Asymptotic Enumeration Methods}, 
  booktitle = {Handbook of Combinatorics}, 
  year = {1995},
  volume = {2},
  editor = {Graham, R.L. and Groetschel, M. and Lovasz, L.},
  publisher = {Elsevier},
  pages = {1063--1229},
  language = {english},
}

@article{Rot76,
  author = {Rothaus, O.S.},
  title = {On ``bent'' functions},
  joutnal = {J. Comb. Theory},
  volume = {A 20},
  year = {1976}, 
  pages = {300--305},
  language = {english},
}

@misc{Sza16,
  title = {A Simple Wide Range Approximation of Symmetric Binomial Distributions},
  author = {Szabados, T.},
  year = {2016},
  eprint = {1612.01112},
  archivePrefix = {arXiv},
  primaryClass = {math.PR},
  language = {english},
}

@book{Tak18,
  title = {A Pythagorean Introduction to Number Theory. Right
    Triangles, Sums of Squares, and Arithmetic},
  author = {Takloo-Bighash, R.},
  year = {2018},
  publisher = {Springer},
  address = {Cham},
  language = {english},
}

\end{sloppypar}
\end{document}